\documentclass[a4paper]{amsart}
  \usepackage{amsmath, amsfonts, amssymb, amsthm}
  \usepackage{enumerate}

\theoremstyle{definition}
\newtheorem{cor}{Corollary}[section]
\newtheorem{dfn}[cor]{Definition} 
\newtheorem{ex}[cor]{Example}
\newtheorem{lem}[cor]{Lemma}
\newtheorem{prop}[cor]{Proposition}
\newtheorem{ques}[cor]{Question}
\newtheorem{rem}[cor]{Remark}

\newcommand{\Hq}{\mathbb{H}}

\newcommand{\Q}{\mathbb{Q}}
\newcommand{\R}{\mathbb{R}}
\newcommand{\Z}{\mathbb{Z}}

\newcommand\explanation[1]{%
  \ifhmode\unskip\fi\hspace*{1em plus 1fill}%
\parbox[t]{6cm}{#1}}

\newcounter{enumcounter}

\usepackage{tkz-graph}

\begin{document}

\title{Towards a Chevalley-style non-commutative algebraic geometry}
\author{Nikolaas D. Verhulst}
\address{TU Dresden\\
	Fachrichtung Mathematik\\
	Insitut f\"ur algebra\\
	01062 Dresden, Germany}
	\email{nikolaas\_damiaan.verhulst@tu-dresden.de}
\maketitle

\begin{abstract}
We aim to construct a non-commutative algebraic geometry by using generalised valuations. To this end, we introduce groupoid valuation rings and associate suitable value functions to them. We show that these objects behave rather like their commutative counterparts. Many examples are given and a tentative connections with Dubrovin valuation rings is established.
\end{abstract}

\section*{Introduction}

Although valuation theory has reached a venerable age by now, it is still fertile ground for new research, in connection with e.g. resolution of singularities or tropical geometry. In the commutative case, the theory of valuations acts as a kind of translation mechanism, turning a field extension of transcendence degree one\footnote{If the transcendence degree is higher things become more complicated, but valuation theory is still useful e.g. to verify properness of maps.} into its Zariski-Riemann surface and associating to a curve a collection of valuation rings in a field. We refer to Chevalley's classic work \cite{Chevalley} for an in-depth study of this correspondence. It seems reasonable to hope that, with a good non-commutative analogue of a valuation ring, one could mimic the commutative construction to arrive at non-commutative abstract Riemann surfaces and general non-commutative algebraic geometry. 

Over the years, many suggestions have been made for what this \textit{good analogue} could be. Besides Schilling's original definition of a non-commutative valuation ring (see below), the most important candidates are perhaps Dubrovin valuation rings (cfr. e.g. \cite{MMU} or \cite{theBook}) and gauges (cfr. \cite{theOtherBook}).

 The classical definition of a valuation ring as a subring $R$ of a field $k$ such that 
\begin{enumerate}[(1)]
	\item $\forall x\in k^{*}: x\notin R\Rightarrow x^{-1}\in R$ \explanation{($R$ is \emph{total})}
	\setcounter{enumcounter}{\value{enumi}}
\end{enumerate}
can easily be adapted for skewfields by adding the condition
\begin{enumerate}[(1)]
	\setcounter{enumi}{\value{enumcounter}}
	\item $\forall x\in k^{*}: xRx^{-1}=R$\explanation{($R$ is \emph{stable}).}
\end{enumerate}
Schilling has shown that, if these two conditions hold true in a given ring $R$, then one can associate to $R$ an equivalence class of valuation functions in much the same way as in the commutative case. Yet this definition is not completely satisfactory. For example, even in a very well-behaved non-commutative extension, valuations on the centre might not extend to the whole skewfield. A $p$-valuation on $\Q$, for instance, does not extend to a valuation on the skewfield of Hamilton quaternions $\Hq$ (unless $p=2$, see \cite{Wadsworth}). 

This problem can be solved by considering $\Hq$ as a $(\Z/2\Z)^{2}$-graded skewfield extension of $\R$. Just like a field is defined as a ring wherein every non-zero element is invertible, so is a graded skewfield defined as a graded ring wherein every \emph{homogeneous} element is invertible. Similarly, a graded valuation ring is a homogeneous subring $R$ of a graded skewfield such that (1) and (2) hold for any \emph{homogeneous} $x\in k^{*}$. With this definition, it is easy to check that, if we write $\R_{p}$ for the ring of positives of some extension of the $p$-adic valuation to $\R$, the ring $\R_{p}\oplus \R_{p}i\oplus \R_{p}j\oplus \R_{p}k$ is a graded valuation ring on $\Hq$ considered as a $(\Z/2\Z)^{2}$-graded skewfield. Graded valuation rings were introduced by Johnson in the late seventies (cfr. \cite{Johnson}) and are still an active field of research (cfr. e.g. \cite{AAC}).

However, in a certain sense, the correct non-commutative counterpart for the notion of a field is not that of a skewfield but rather that of a simple artinian ring and, unfortunately, such rings need not be graded skewfields. We do know, on the other hand, that they can be equipped with a natural groupoid grading since they are isomorphic to matrix rings over skewfields -- and matrix rings are groupoid graded (see section \ref{Basics}). 

In this paper, we will introduce the notions of a $G$-skewfield (of which simple artinian rings will be examples) and a $G$-valuation ring by adapting conditions (1) and (2) to a groupoid-graded context. We will also introduce $G$-valuations -- a natural generalisation of valuations -- and we will show that, as in the commutative case, $G$-valuation rings are in one-one correspondence with equivalence classes of $G$-valuations. Since our versions of conditions (1) and (2) are far less rigid, we will be able to associate $G$-valuations to many relatively ill-behaved subrings as well. Finally, we will establish a tentative connection with Dubrovin valuation rings. 

The next step in the non-commutative algebraic geometry programme hinted at in the first paragraph, would be proving approximation theorems and developing $G$-divisor theory. Once so far, a $G$-graded version of the Riemann-Roch theorem should be within reach; an ambitious goal, perhaps, but a worthy one.

\section{Terminology and basic properties}\label{Basics}

Remember that a \emph{groupoid} is a (small) category wherein all morphisms are invertible. Alternatively, a groupoid can be thought of as a group for which the multiplication is only partially defined. As a concrete example, let $e_{ij}$ be the $n\times n$-matrix with a one on place $i,j$ and zeroes everywhere else. Then $\Delta_{n}=\left\{e_{ij}\mid 1\leq i,j\leq n\right\}$ equipped with the partial multiplication \[(.\cdot.):\Delta_{n}\times\Delta_{n}\to\Delta_{n}: (e_{ij},e_{kl})\mapsto
\begin{cases}
	e_{il}&\text{if $j=k$}\\
	\text{undefined}&\text{if $j\neq k$}
\end{cases}\] is a groupoid. In fact, we can do this more generally: if $G$ is any group, then the set $G\left[\Delta_{n}\right]=G\times\Delta_{n}$ with the partial multiplication \[(.\cdot.):G\left[\Delta_{n}\right]\times G\left[\Delta_{n}\right]\to G\left[\Delta_{n}\right]: ((g,e_{ij}),(g',e_{kl}))\mapsto
\begin{cases}
	(gg',e_{il})&\text{if $j=k$}\\
	\text{undefined}&\text{if $j\neq k$}
\end{cases}\]
is a groupoid as well.

 For the remainder of this paper, $G$ will be a groupoid and $R$ will be a ring. We will use the notation $\textbf{s}(g)=gg^{-1}$ and $\textbf{t}(g)=g^{-1}g$ for the source and the target of $g\in G$ respectively. Note that the multiplication $gg'$ of two elements $g,g'\in G$ is defined if and only if $\textbf{t}(g)=\textbf{s}(g')$. Two elements $g$ and $g'$ of $G$ are called \emph{connected} if there is a morphism from $\textbf{t}(g)$ to $\textbf{s}(g')$. This is a reflexive and transitive property which, since $G$ is a groupoid, is also symmetric. The connected components are the equivalence classes with respect to connectedness, i.e. the maximal subsets of $G$ in which any two elements are connected.

\begin{dfn}
$R$ is said to be \emph{$G$-graded} if there are abelian subgroups $(R_{g})_{g\in G}$ such that $R=\bigoplus_{g\in G} R_{g}$ and $R_{g}R_{g'}\subseteq R_{gg'}$ if $gg'$ exists while $R_{g}R_{g'}=0$ otherwise. If $R_{g}R_{g'}=R_{gg'}$ whenever $gg'$ exists, then the grading is called \emph{strong}.
\end{dfn}

\begin{ex}
The \emph{groupoid ring} $R[G]$ is constructed by endowing the set \[R[G]=\left\{f:G\to R\mid \#\left\{g\in G\mid f(g)\neq 0\right\}<\infty\right\},\] with a sum and a multiplication as follows:\[ (f+f')(g)=f(g)+f'(g),\quad(ff')(g)=\sum_{g'g''=g}f(g')f'(g'').\] These operations are well-defined since $f$ and $f'$ have finite support. In a similar fashion as for group rings it can be checked that they define a ring structure on $R[G]$. This ring is canonically $G$-graded by putting \[R[G]_{g}=\left\{f:G\to R\mid \forall g'\neq g:f(g')=0\right\}\] for all $g\in G$. The most important example of groupoid graded rings are matrix rings: $M_{n}(R)$ is isomorphic to $R[\Delta_{n}]$.
\end{ex}

An element $h$ is in a groupoid-graded ring $R=\bigoplus_{g\in G}R_{g}$ is called \emph{homogeneous} if it is in $\bigcup R_{g}$. An ideal or a subring is called homogeneous if it is generated by homogeneous elements. We will call a homogeneous ideal $G$-maximal if it is maximal among proper homogeneous ideals. Similarly, we will call a $G$-graded ring $G$-simple if it contains no proper homogeneous ideals. The \emph{support} of an element $r=\sum_{g\in G} r_{g}$ is the set of $g\in G$ for which $r_{g}\neq 0$. The support of a set is the union of the supports of its elements.

We use $G_{0}$ for the \emph{principal component}, i.e. the set of idempotent elements of $G$. It is harmless to assume that $G_{0}$ consists of but finitely many elements and that, if $1=\sum_{e\in G_{0}}1_{e}$ is the homogeneous decomposition of $1$, we have $1_{e}\neq0$  for all $e\in G_{0}$ (cfr. \cite{Lundstrom}).

\begin{prop}
If $R$ is $G$-graded, then the following elementary properties hold:
\begin{enumerate}[(1)]
	\item $R_{e}$ is a ring for any idempotent $e$ of $G$.
	\item If $I$ is a $G$-ideal of $R$, then $I_{e}$ is an ideal of $R_{e}$ for every idempotent $e$.
	\item $R_{g}$ is a left $R_{\textbf{s}(g)}$, right $R_{\textbf{t}(g)}$-module.
	\item $G$ is a group if and only if there is some invertible homogeneous element.
\end{enumerate}
\end{prop}
\begin{proof}
$R_{e}$ is by definition closed under addition and, since $e$ is an idempotent, it is also closed under multiplication. Since the product of two distinct idempotents $e$ and $e'$ of $G$ is never defined, we have $r=r1=\sum_{e\in G_{0}}r1_{e}=r1_{e}$ for all $r\in R_{e}$. Hence $1_{e}$ is the unit of $R_{e}$. For (3) it suffices to note that the map \[(.\cdot .):R_{\textbf{s}(g)}\times R_{g}\to R_{g}:(x,y)\mapsto xy\]defines a left $R_{\textbf{s}(g)}$-multiplication on $R_{g}$, the right $R_{\textbf{t}(g)}$-multiplication being defined analogously. (2) is a special case of (3) in disguise where $I_{e}\subseteq R_{e}$. To prove (4), note that, since $ee'$ is undefined for idempotents $e\neq e'$, any homogeneous element $h\in R_{g}$ must be a zero divisor if there is some unit $e\neq \textbf{t}(h)$ or $e\neq \textbf{s}(h)$. If $G$ is a group with unit $e$, then $1\in R_{e}$ is homogeneous and invertible.
\end{proof}

\begin{prop}
Let $R$ be a strongly $G$-graded ring. The homogeneous ideals of $R$ are in 1-1 correspondence with ideals \[I_{e_{1}}\subseteq R_{e_{1}},...,I_{e_{n}}\subseteq R_{e_{n}}\] where the $e_{i}$ are representatives of the connected components of $G$. 
\end{prop}
\begin{proof}
Suppose $g$ is in the connected component of $e$, i.e. $g=g'eg''$. Since $R$ is strongly graded, we must have that \[I_{g}=R_{g'}R_{g'^{-1}}I_{g}R_{g''^{-1}}R_{g''}\subseteq R_{g'}I_{e}R_{g''}\subseteq I_{g'eg''}=I_{g}\] so any two homogeneous ideals of $R$ restricting to the same ideals on $R_{e_{1}},...,R_{e_{n}}$ must be equal. Suppose, on the other hand, if $I_{e_{1}}\subseteq R_{e_{1}},...,I_{e_{n}}\subseteq R_{e_{n}}$ are ideals in their respective rings. If $g$ is in the same connected component as $e$, then we can define $I_{g}=R_{g'}I_{e}R_{g''}$ where $g'$ and $g''$ are connecting elements for $g$ and $e$. $I=\sum_{g\in G} I_{g}$ is then a homogeneous ideal of $R$.
\end{proof}

As an immediate consequence of the preceding proposition, the $G$-maximal ideals of a strongly $G$-graded $R$ are those corresponding to a maximal ideal in one of the connected components and to $R_{g}$ for any $g$ not in that component. Therefore, the intersection of the $G$-maximal ideals -- which we call the $G$-Jacobson radical -- is the homogeneous ideal corresponding to the Jacobson radical in every connected component.

\section{$G$-skewfields}

 We write, for any $a$ in a $G$-graded $R$, 
\[
	\textbf{t}(a)=\sum_{\substack{e\in G_{0}\\ a1_{e}\neq 0}}1_{e}\quad\text{and}\quad \textbf{s}(a)=\sum_{\substack{e\in G_{0}\\ 1_{e}a\neq 0}}1_{e}.
\]
A $G$-inverse of $a$ is an element $b$ satisfying
\[
	\textbf{s}(a)=ab=\textbf{t}(b)\quad\text{and}\quad	\textbf{s}(b)=ba=\textbf{t}(a).
\]
If $a$ has a $G$-inverse, we say that it is $G$-invertible. We will use the notation $a^{-1}$ for the $G$-inverse of $a$, but one should keep in mind that the $G$-inverse of $a$ may exist even if $a$ is not invertible in $R$. In a desperate attempt to avoid confusion, we will denote the set of $G$-invertible elements of a $G$-graded ring $R$  by $R^{*}$, while the set of invertible elements will be denoted by $U(R)$.

\begin{prop}
If $R$ is $G$-graded, then:
\begin{enumerate}[(1)]
	\item The $G$-inverse of $a\in R$, if it exists, is unique.
	\item $(ab)^{-1}=b^{-1}a^{-1}$ if all terms involved exist.
	\item The $G$-inverse of $a\in R_{h}$, if it exists, is in $R_{h^{-1}}$.
	\item If $a$ is invertible in $R$, say $ba=ab=1$, then $b$ is the $G$-inverse of $a$.
	\item The grading on $R$ is strong if and only if $R_{g}\cap R^{*}\neq\emptyset$ for all $g\in G$.
\end{enumerate}
\end{prop}
\begin{proof}
If $b$ and $b'$ are $G$-inverses of $a$, then 
	\begin{equation*}
		b=b\textbf{t}(b)=b\textbf{s}(a)=bab'=\textbf{s}(a)b'=b'ab'=b'\textbf{t}(b')=b'
	\end{equation*}
which proves (1). (2) is obvious. Suppose $a\in R_{h}$ is homogeneous and let $a^{-1}=\sum_{g\in G} b_{g}$. For all $g\neq h^{-1}$ we have that $aa^{-1}=\textbf{s}(a)$ implies $ab_{g}=0$ and $a^{-1}a=\textbf{t}(a)$ implies $b_{g}a=0$. Therefore, $b_{h^{-1}}$ is a $G$-inverse and by (1) it must be unique. If $a$ is invertible with inverse $b$, then $1_{e}a\neq 0$ for all $e\in G_{0}$, which establishes that $\textbf{s}(a)=1$. Similarly, we find $\textbf{t}(a)=1$ and by symmetry the same holds for $b$. Consequently, $a$ and $b$ are each others $G$-inverses. To show (6), suppose that $R_{g}\cap R^{*}\neq\emptyset$ for all $g\in G$ and assume $gg'$ exists. Then we have \[R_{gg'}= R_{\textbf{s}(g)}R_{gg'}= R_{g}R_{g^{-1}}R_{gg'}\subseteq R_{g}R_{g'}\subseteq R_{gg'}\]so the grading is strong. If we assume the grading to be strong, then $R_{g}R_{g^{-1}}$ must contain $1_{\textbf{s}(g)}$ which implies that some element of $R_{g}$ is $G$-invertible.
\end{proof}

A (group) graded ring is called a (group) grade skewfield if the homogeneous elements form a group (cfr. \cite{Methods}). Similarly, we will call a $G$-graded ring a \emph{$G$-skewfield} if the homogeneous elements form a groupoid, in other words, if every homogeneous element is $G$-invertible. In view of the preceding proposition, $G$-skewfields are necessarily strongly graded.

\begin{ex}
For a (skew)field $k$, the groupoid ring $k[G]$ is a $G$-(skew)field. This means in particular that the matrix ring $M_{n}(k)$ is a $\Delta_{n}$-skewfield.
\end{ex}

\begin{prop}
If $Q$ is a $G$-skewfield, then
\begin{enumerate}[(1)]
	\item If $Q$ is a $G$-skewfield, then $Q_{e}$ is a skewfield for any idempotent $e$.
	\item If $Q$ is a $G$-skewfield then, for non-zero $h\in Q_{g},h'\in Q_{g'}$, $hh'=0$ if and only if $\textbf{t}(g)\neq\textbf{s}(g')$. 
\end{enumerate}
\end{prop}
\begin{proof}
Since for any $a\in Q_{g}$ we have $a^{-1}\in Q_{g^{-1}}$ (1) follows. Assume that (2) does not hold, then we can take non-zero $h,h'$ with $hh'$. Then $0=h^{-1}hh'=1_{\textbf{t}(h)}h'=h'$ which is a contradiction.
\end{proof}

\begin{prop}
A $G$-skewfield is $G$-simple in the sense that it has no non-trivial homogeneous ideals if and only if $G$ is connected.
\end{prop}
\begin{proof}
Clearly, $G$ being connected is a necessary condition for a $G$-skewfield to have no homogeneous ideals since $\bigoplus_{g\in C}Q_{g}$ for a connected component $C\subseteq G$ is always a homogeneous ideal. On the other hand, if every element is $G$-invertible, then an ideal $I$ with a homogeneous $a\in I\cap R_{h}$ necessarily contains all $1_{e}$ for $e\in G_{0}$ in the connected component of $h$, so a connected $G$-skewfield is $G$-simple.
\end{proof}

\begin{ex}
Let $k$ be a field and let $G$ be $\Delta_{2}\left[\Z/2\Z\right]$.
Then $Q=k[G]$ is an example of a $G$-skewfield for which $\mathrm{supp}(Q)$ is connected (so it is $G$-simple) but which is not simple. Indeed, $S_{(e_{11},0)}\simeq S_{(e_{22},0)}\simeq k[\Z/2\Z]$ and this ring contains non-trivial ideals.
\end{ex}

The first building block is firmly in place, now: $G$-graded skewfields will play the same role that fields play in classical valuation theory. The most important examples are of course the matrix rings, but there are more as the following construction shows. Suppose $k$ is a field and assume that a partial function $\alpha:G\times G\to k^{*}$ and a map $\sigma:G\to \mathrm{Aut}(k)$ have been given such that for all $a\in k$ and $f,g,h\in G$
\begin{enumerate}
	\item $\sigma(f)(\sigma(g)(a))=\alpha(f,g)\sigma(fg)(a)\alpha(f,g)^{-1}$,
	\item $\alpha(f,g)\alpha(fg,h)=\sigma(f)(\alpha(g,h))\alpha(f,gh)$,
	\item $\alpha(f,\textbf{t}(f))=1=\alpha(\textbf{s}(f),f)$,
	\item $\alpha(f,g)$ exists if $fg$ exists.
\end{enumerate}
Let $k[G,\alpha,\sigma]$ denote the free $k$-module with basis $G$ and define
a multiplication by demanding \[(ag)(bh)=\begin{cases}
	a\sigma(g)(b)\alpha(f,g)fg&\text{if $fg$ is defined}\\
	0&\text{otherwise}
\end{cases}\]
and distributivity.

\begin{prop}
This is indeed a $G$-skewfield and, if $G$ is connected, then every $G$-skewfield $Q$ is of this form.
\end{prop}
\begin{proof}
To show the latter statement, notice first that, due to the connectedness of $G$, we have $Q_{e}\simeq Q_{e'}$ for any $e,e'$ idempotent. Choose such an isomorphism and call it $\iota_{e,e'}$. Take for any $g\in G$ a $G$-invertible $u_{g}\in Q_{g}$. We assume $u_{e}$ to be the identity of $Q_{e}$ for any $e\in G_{0}$. Define a map $\sigma:G\to\mathrm{Aut}(Q_{e})$ by putting \[\sigma(g):k\to k:a\mapsto \iota_{\textbf{s}(g),e}\left(u_{g}\iota_{e,\textbf{t}(g)}(a)u_{g}^{-1}\right)\]
and a partial function $\alpha:G\times G\to k^{*}$ by
\[\alpha(f,g)=\begin{cases}
	u_{f} u_{g}u_{fg}^{-1}&\text{if $fg$ is defined}\\
	\text{undefined}&\text{otherwise}
\end{cases}\]
To check that these functions satisfy the necessary conditions, that $Q$ is isomorphic to $Q_{e}[G,\alpha,\sigma]$, and that any $k[G,\alpha,\sigma]$ is indeed a $G$-skewfield, it suffices to sprinkle the phrase "if $fg$ is defined" liberally throughout the group-graded proof from \cite{Methods}. Since this is relatively straightforward but rather tedious we omit it here.
\end{proof}

\begin{ex}
Take a proper field extension $k\hookrightarrow k(\sqrt{a})$, then the $G$-skewfield \[Q=\begin{pmatrix}
	k&\sqrt{a}k\\
	\sqrt{a}k&k
\end{pmatrix}\]
is by Artin-Wedderburn isomorphic to $M_{2}(k)$ but, if both rings are endowed with their respective canonical $G$-gradings, not as a $G$-graded ring. This is an example of a (non-trivially) twisted groupoid ring. 
\end{ex}

\begin{prop}
A $G$-skewfield $Q$ is artinian if and only if all $S_{e}$ are artinian.
\end{prop}
\begin{proof}
If some $S_{e}$ is not artinian, then there exists an infinite descending chain $I_{0,e}\supsetneq I_{1,e}\supsetneq\cdots$ of $S_{e}$-ideals. This induces a chain $I_{0}\supsetneq I_{1}\supsetneq\cdots$ of $Q$-ideals by putting $I_{n}=\bigoplus _{\textbf{t}(g)=e=\textbf{s}(g')}h_{g}I_{e,n}h_{g'}\oplus\bigoplus_{g\notin C_{e}}Q_{g}$ where $C_{e}$ is the connected component of $e$ and $h_{g}$ and $h_{g'}$ are arbitrary non-zero elements in $Q_{g}$ and $Q_{g'}$ respectively.\footnote{This is a slight abuse of notation, since the sum $\bigoplus _{\textbf{t}(g)=e=\textbf{s}(g')}h_{g}I_{e,n}h_{g'}$ is only direct up to repetitions.} 

Suppose on the other hand that all $S_{e}$ are artinian and that $I_{0}\supsetneq I_{1}\supsetneq\cdots$ is an infinite chain of descending ideals in $Q$. Then, for any $e\in G_{0}$, $1_{e}I_{0}1_{e}\supseteq 1_{e}I_{1}1_{e}\supseteq\cdots$ gives a descending chain of ideals in $S_{e}$. Such a chain must stop, so there is some $n$ with $1_{e}I_{n}1_{e}=1_{e}I_{n+1}1_{e}=\cdots$ for all $e\in G_{0}$. Take $x\in I_{n}\setminus I_{n+1}$, then $1_{e}x1_{e'}\notin I_ {n+1}$ for some $e,e'\in G_{0}$. In fact, these $e$ and $e'$ are in the same connected component, so $S_{e}\simeq S_{e'}$ and $1_{e}x1_{e'}h1_{e}\notin I_{e}$ where $h$ is an arbitrary homogeneous element connecting $e'$ and $e$.
\end{proof}

\section{$G$-valuation rings and $G$-valuations}

For the remainder of this section, we let $R$ be a $G$-graded subring of a $G$-skewfield $Q$. If for every homogeneous $h\in Q$ we have either $h\in R$ or $h^{-1}\in R$, then we say that $R$ is \emph{$G$-total}\index{\textit{G}-total}. This is the canonical generalisation of totality (as referred to in the introduction) and gives rise to somewhat similar results. Note that if $R$ is a $G$-total subring of the $G$-skewfield $Q$, then $R_{e}$ is a total subring of the skewfield $Q_{e}$ for any idempotent $e\in G$. This implies that any $G$-total subring of a $G$-skewfield contains $1_{e}$ for all idempotents $e$.

\begin{prop}\label{totord}
Suppose $R$ is $G$-total. If $I$ and $J$ are homogeneous left (resp. right) ideals, then $J_{g}\nsubseteq I_{g}$ implies $I_{g'}\subseteq J_{g'}$ for any $g'$ with the same right (resp. left) unit as $g$. In particular, we have $I_{g}\subseteq J_{g}$ or $J_{g}\subseteq I_{g}$.
\end{prop}
\begin{proof}
Suppose $I$ and $J$ are homogeneous and $J_{g}\nsubseteq I_{g}$, so there exists some non-zero $h\in J_{g}\setminus I_{g}$. Suppose $\textbf{t}(g')=\textbf{t}(g)$, and assume $h'\neq 0$ is in $I_{g'}$ (if no such $h'$ exists the claim is certainly true). This means that $hh'^{-1}$ and $h'h^{-1}$ are defined and at least one of these is in $R$. If $hh'^{-1}$ is in $R$, then $hh'^{-1}h'$ is in $I\cap R_{g}=I_{g}$ which is a contradiction, so $h'h^{-1}$ must be in $R$ and consequently $h'h^{-1}h$ is in $J\cap R_{g'}=J_{g'}$. The other case is similar.
\end{proof}

\begin{cor}
If $R$ is a $G$-total subring, then any left (resp. right) ideal generated by homogeneous elements $h_{1},...,h_{n}$ with the same target (resp. source) is cyclic.
\end{cor}

If $G$ happens to be group, then the previous statements reduce to \textit{ideals are totally ordered} and \textit{finitely generated ideals are cyclic} respectively, both well-known results from (non-commutative) valuation theory which are generally not true in the $G$-graded case as the following example demonstrates.

\begin{ex}\label{Gvalex}
Let $k$ be a field and let $R_{v}$ be a valuation ring in $k$ with unique maximal ideal $\mathfrak{m}_{v}$. Consider the subring 
\[
	R=\begin{pmatrix}
		R_{v}&k\\
		0& R_{v}
	\end{pmatrix}
\] 
of the $\Delta_{2}$-skewfield $Q=M_{2}(k)$. $R$ is a $G$-total subring of $Q$, but the homogeneous ideals are not totally ordered since
\[
I=\begin{pmatrix}
		\mathfrak{m}_{v}&k\\
		0& R_{v}
	\end{pmatrix}\quad\text{and}\quad
	J=\begin{pmatrix}
		R_{v}&k\\
		0& \mathfrak{m}_{v}
	\end{pmatrix}
\]
are incomparable. Note that the fact that $I_{1,1}\subsetneq J_{1,1}$ and $J_{2,2}\subsetneq I_{2,2}$ implies  $I_{1,0}=J_{1,0}$ as well as $J_{1,0}=J_{0,1}$. Note also that the ideal generated by $e_{11}$ and $e_{22}$ is not cyclic.
\end{ex}

\begin{prop}\label{MaxId}
Let $R$ be $G$-total, and put $M$ the (homogeneous) ideal generated by the set of homogeneous elements which are not $G$-invertible in $R$. Then $R/M$ is a $G$-skewfield and $M$ is the maximal homogeneous ideal with the property that it contains no $1_{e}$ for $e\in G_{0}$. 
\end{prop}
\begin{proof}
If $\overline{x}\neq\overline{0}$ is some homogeneous element of $R/M$, then $x=h+p$ where $h$ is a non-zero homogeneous element of $R\setminus M$ and $p\in M$. Let $p=h_{1}+\cdots +h_{n}$ be the homogeneous decomposition of $p$. Then $h^{-1}$ is also in $R\setminus M$ and $\overline{x}\overline{h^{-1}}=\overline{hh^{-1}+ph^{-1}}=\overline{1_{\textbf{s}(h)}}=1_{\textbf{s}(\overline{x})}$ since $ph^{-1}$ must be in $M$ -- otherwise some $h_{i}h^{-1}$ is not in $M$ and we would have $hh_{i}^{-1}\in R\setminus M$ and consequently $h^{-1}hh_{i}^{-1}=h_{i}^{-1}\in R\setminus$, which is a contradiction. Analogously, we find $\overline{h^{-1}}\overline{x}=1_{\textbf{t}(\overline{x})}$ which implies that $R/M$ is a $G$-skewfield. If $M'$ is an ideal which contains $M$ strictly, then there is some homogeneous $h\in M'\setminus M$ so $h$ is $G$-invertible in $R$ and consequently $hh^{-1}$ is in $M$, which implies that $1_{e}\in M$ for some $e\in G_{0}$.
\end{proof}

$R$ will be called \emph{$G$-stable}\index{\textit{G}-stable} if $hR_{\textbf{t}(h)}h^{-1}= R_{\textbf{s}(h)}$ for any homogeneous $h$.  This implies that $R_{e}$ is stable for all $e\in G_{0}$. In particular, if $R$ is a $G$-total $G$-stable subring of the $G$-skewfield $Q$, then $R_{e}$ is a graded valuation ring in $Q_{e}$ for every $e\in G_{0}$.

\begin{prop}
Any homogeneous right (resp. left) ideal of $R$ is a left (resp. right) ideal if $R$ is $G$-stable.
\end{prop}
\begin{proof}
Let $I$ be a right $G$-ideal of $R$, let $h\in I\cap Q_{g}$ be homogeneous and pick $r\in R\cap Q_{g'}$ arbitrary. If $g'g$ does not exist $rh=0\in I$ follows, so suppose $g'g$ does exist. Because of $G$-stability of $R$, $h^{-1}rh$ is in $R$ -- whether it is zero or not. If $h^{-1}r$ exists, we have $hh^{-1}rh=rh\in I$ since $I$ is a right $G$-ideal. If $h^{-1}r$ does not exist, $hh^{-1}rh=0$ so it is again in $I$. This proves the claim for right $G$-ideals; the reasoning for left $G$-ideals is similar.
\end{proof}

\begin{dfn}
If $R$ is $G$-total and $G$-stable, we call it a \emph{$G$-valuation} ring.
\end{dfn}

With that, the second important concept is in place. Next on the menu are the $G$-valuation functions but, as a small intermezzo, we will consider some examples first: if $R_{v}$ is a valuation ring in a skewfield $D$, then $M_{n}(R_{v})$ is a $\Delta_{n}$-valuation ring in the $\Delta_{n}$-skewfield $M_{n}(D)$. This already yields a vast class of examples of $G$-valuation rings but there are many more, like example \ref{Gvalex} or the following example.

\begin{ex}
Consider the rational Hamilton quaternions $\Hq(\Q)$. There is a natural $(\Z/2\Z)^{2}$-grading on $\Hq(\Q)$ and consequently $M_{n}(\Hq(\Q))$ is a $(\Z/2\Z)^{2}[\Delta_{n}]$-skewfield. Any $p$-valuation on $\Q$ extends to a graded valuation ring $R=\Z_{p}\oplus \Z_{p} i\oplus \Z_{p} j\oplus \Z_{p} k$ in $\Hq(\Q)$. $M_{n}(R)$ is then a $(\Z/2\Z)^{2}[\Delta_{n}]$-valuation ring on $M_{n}(\Hq(\Q))$. 
\end{ex}

We say that a groupoid $G$ is \emph{partially ordered} by some partial order relation $\leq$ if $g\leq g'$ implies $hg\leq hg'$ and $gh\leq g'h$ when the multiplications are defined. We will say that $G$ is \emph{ordered} if every $g\in G$ is comparable to $\textbf{s}(g)$ and $\textbf{t}(g)$. If $G$ is a group, ordered in this sense is the same as totally ordered.

\begin{dfn}
Let $G$ be a groupoid and let $Q$ be a $G$-skewfield. A \emph{$G$-valuation} on $Q$ is a surjective map $v:Q\to\Gamma\cup\left\{\infty\right\}$ for some ordered groupoid $\Gamma$ (with, as usual, $\infty>\gamma$ and $\infty\gamma=\gamma\infty=\infty$ for all $\gamma\in\Gamma$) satisfying:
\begin{enumerate}
	\item $v(x)=\infty\Leftrightarrow x=0$,
	\item $v(x+y)\geq v(z)$ if $v(y)\geq v(z)\leq v(x)$,
	\item $v(hh')=v(h)v(h')$ for $h\in Q_{g},h'\in Q_{g'}$ if $gg'$ is defined.
	\setcounter{enumcounter}{\value{enumi}}
\end{enumerate}
\end{dfn}

We have all the components now, but we still have to make sure everything fits smoothly together. If $v:Q\to\Gamma\cup\left\{\infty\right\}$ is a $G$-valuation, then we let $T_{v}$ be the ring generated by homogeneous elements $h$ with $v(h)\geq v(\textbf{t}(h))$. Note that, since $G_{0}$ is a finite set, $1\in T_{v}$ follows. Since $T_{v}$ is generated by homogeneous elements, it inherits the $G$-grading from $Q$. 

\begin{prop}\label{G-st&G-t}
For any $G$-valuation $v:Q\to\Gamma\cup\left\{\infty\right\}$ the ring $T_{v}$ is $G$-stable and $G$-total.
\end{prop}
\begin{proof}
Suppose $h$ is a homogeneous element of $Q$ and suppose $v(h)<v(\textbf{t}(h))$. Then 
\begin{align*}
	v(h^{-1})=v(\textbf{t}(h))v(h^{-1})&>v(h)v(h^{-1})=v(\textbf{s}(h))=v(\textbf{t}(h^{-1}))
\end{align*}
showing that $h^{-1}\in T_{v}$. To show $G$-stability, pick some homogeneous element $h\in Q_{g}$ and suppose that $r\in T_{\textbf{t}(g)}$, then $v(hrh^{-1})\geq v(h)v(1_{\textbf{t}(h)})v(h^{-1})=v(1_{\textbf{s}(h)})$. Since $\textbf{s}(h)$ is the target of $hrh^{-1}$, this shows that $hrh^{-1}\in T\cap Q_{\textbf{s}(h)}=T_{\textbf{s}(h)}$. Similarly, if $r'$ is in $T_{\textbf{s}(g)}$, it follows that $h^{-1}r'h$ is in $T_{\textbf{t}(h)}$, so $r'\in hT_{\textbf{t}(h)}h^{-1}$ which establishes the $G$-stability of $T_{v}$.
\end{proof}

Obviously, one can define another ring, $S_{v}$ say, as the ring generated by homogeneous elements with $v(h)\geq v(\textbf{s}(h))$. Mutatis mutandis, proposition \ref{G-st&G-t} can be proven for $S_{v}$ instead of $T_{v}$. 

\begin{ex}
The groupoid $\Delta_{2}$ can be ordered by letting $e_{11}$ be the maximum and $e_{22}$ the minimum of $\Delta_{2}$, the elements $e_{12}$ and $e_{21}$ remaining incomparable. Using this, we can define an ordering on the groupoid $G$ of non-zero elements of $\Z[\Delta_{2}]$ by putting 
\[\sum_{\delta\in \Delta_{2}}a_{\delta}\delta\geq\sum_{\delta\in \Delta_{2}}b_{\delta}\delta\quad\text{ if, for all $\delta\in\Delta_{2}$,}\quad
a_{\delta}\geq b_{\delta}\text{ or }\exists \delta'>\delta:a_{\delta'}\geq b_{\delta'}.\]
If $v$ is a discrete valuation on a field $k$, then \[v:M_{2}(k)\to G\cup\left\{\infty\right\}:\sum_{\delta\in \Delta_{2}}m_{\delta}\mapsto \sum_{\substack{\nexists \delta'<\delta\\ m_{\delta'}\neq0}}v(m_{\delta})e_{\delta}\] is a groupoid valuation. In this case, we have \[T_{v}=\begin{pmatrix}R_{v}&k\\0&R_{v}\end{pmatrix}\quad\text{while}\quad S_{v}=\begin{pmatrix}R_{v}&0\\k&R_{v}\end{pmatrix}.\]
\end{ex}

We will now show that, although $T_{v}$ and $S_{v}$ might be different, there exists for any $G$-valuation ring $R$ some $G$-valuation $v$ with $T_{v}=R=S_{v}$. First, we briefly remind the reader how quotients of groupoids can be defined. Suppose $F$ is subgroupoid of $G$ containing all idempotents and such that $gF_{\textbf{t}(g)}g^{-1}=F_{\textbf{s}(g)}$ for all $g$, then one can construct a factor groupoid $G/F=G/\sim$ where
\begin{equation*}
	g\sim h\quad\Leftrightarrow\quad \exists f_{s},f_{t}\in F: g=f_{s}hf_{t}.
\end{equation*}
It can easily be verified that this is an equivalence relation which is compatible with the multiplication on $G$, so there is a canonical induced multiplication on $G/F$.

\begin{prop}\label{Gval}
For any $G$-stable, $G$-total subring $R$ there is some ordered groupoid $\Gamma$ and some partial $G$-valuation $v:Q\to\Gamma\cup\left\{\infty\right\}$ with $T_{v}=R=S_{v}$.
\end{prop}
\begin{proof}
$H(Q)^{*}$ is a groupoid for the multiplication and $H(R)^{*}$ is a subgroupoid containing all $1_{e}$ for $e\in G_{0}$, which are exactly the idempotents of $H(Q)^{*}$. Moreover, because of the $G$-stability of $R$, we have $hH(R)^{*}_{\textbf{t}(h)}h^{-1}=H(R)^{*}_{\textbf{s}(h)}$ for all $h\in H(Q)^{*}$.  Denote the quotient groupoid $H(Q)^{*}/H(R)^{*}$ by $\Omega$. We can define an ordering on $\Omega$ by 
\[
	\overline{x}\geq\overline{y}\Leftrightarrow\exists r_{s},r_{t}\in H(R): x=r_{s}yr_{t}.
\]
It is a standard verification that this is a well-defined partial order relation relation on $\Omega$. Pick $\omega=\overline{q}$ in $\Omega$ for some $q\in H(Q)^{*}$. If $q\in H(R)$, then $\omega=\overline{q}\overline{1_{\textbf{t}(q)}}$ and $\overline{1_{\textbf{t}(q)}}=\textbf{t}(\omega)$ so $\omega$ is comparable to its target. If $q\notin H(R)$, then $q^{-1}\in H(R)$ because of the $G$-totality of $R$. Therefore, $\textbf{t}(\omega)=\overline{q^{-1}}\overline{q}$ so $\omega$ is again comparable to its target. Of course, a similar argument holds for sources instead of targets. Suppose now $\chi,\psi,\omega\in\Omega$ such that $\psi\leq\omega$ and both $\chi\psi$ and $\chi\omega$ are defined. We have $x,y,z\in H(Q)^{*}$ and $r_{s},r_{t}\in H(R)\setminus\left\{0\right\}$ with $\overline{x}=\chi,\overline{y}=\psi,\overline{z}=\omega$ and $r_{s}yr_{t}=z$. Clearly, $\psi$ and $\omega$ have the same source, so we can assume $y$ and $z$ to have equal source as well. Therefore $\textbf{t}(r_{s})=\textbf{s}(r_{s})$ holds and, by the $G$-stability of $R$, we have $yR_{\textbf{t}(y)}y^{-1}=R_{\textbf{s}(y)}$ so there is some $r_{t}'\in R_{\textbf{t}(y)}$ with $r_{s}y=yr_{t}'$. Thus $xz=xyr_{t}'r_{t}$ hence \[\chi\psi=\overline{xy}\leq\overline{xz}=\chi\omega.\]The other compatibilities can be checked in a similar fashion. This means $\Omega$ is ordered. Note that if $h$ and $h'$ are in the same $Q_{g}$ they must be comparable and $\overline{h+h'}\geq\min\left\{\overline{h},\overline{h'}\right\}$.

Let $g\sim g'$ if for any $h_{g'}\in Q_{g'}^{*}$ there are $h_{g},h_{g}'\in Q_{g}^{*}$ with $\overline{h_{g}}\leq\overline{h_{g'}}\leq\overline{h_{g}'}$. This is an equivalence relation compatible with multiplication, so $\overline{G}=G/\sim$ is a groupoid which can be ordered by putting for all $\overline{g},\overline{g'}\in \overline{G}$, $\overline{g}<\overline{g'}$ if and only if $\overline{h}< \overline{h'}$ for all non-zero $h\in R_{g}$, $h'\in R_{g'}$. Set $\Gamma$ the groupoid of non-zero elements of $\Omega[\overline{G}]$ ordered by \[\sum_{\overline{g}\in \overline{G}}a_{\overline{g}}\overline{g}\geq\sum_{\overline{g}\in \overline{g}}b_{\overline{g}}\overline{g}\quad\text{ if, for all $\overline{g}\in\overline{G}$,}\quad
a_{\overline{g}}\geq b_{\overline{g}}\text{ or }\exists \overline{g'}>\overline{g}:a_{\overline{g'}}\geq b_{\overline{g'}}.\]
Define \[v:Q\to\Gamma\cup\left\{\infty\right\}:\sum_{g\in G}m_{g}\mapsto \sum_{\substack{\nexists \overline{g'}<\overline{g}\\m_{\overline{g'}}\neq0}}\min\left\{\overline{m_{g}}\mid g\in\overline{g}\right\}\overline{g}.\]
This will be our $G$-valuation.

It is clear that $v\left(\sum_{g\in G}m_{g}\right)=\infty$ can only happen if all $m_{g}$ are zero, and it is just as clear that $v(hh')=v(h)v(h')$ if $h$ and $h'$ are homogeneous. Suppose $x$, $y$ and $z$ are such that $v(x)\geq v(z)\leq v(y)$, then $v(x+y)\geq v(z)$ since this property holds in $\Omega$.

We certainly have $R\subseteq T_{v}$ and we know that $1_{e}\in R$ for any idempotent $e\in G_{0}$. Suppose now that $v(h)\geq v(\textbf{t}(h))$ for some homogeneous $h$. If $h$ is not in $R$, then $h^{-1}$ is in $R$, whence $v(h^{-1})\geq v(\textbf{t}(h^{-1}))=v(\textbf{s}(h))$, leading to $v(\textbf{t}(h))=v(h^{-1})v(h)\geq v(\textbf{s}(h))v(h)=v(h)$, i.e. $v(h)=v(\textbf{t}(h))$. This means that $\overline{h}=\overline{1_{\textbf{t}(h)}}$, so there are $r_{s},r_{t}$ in $H(R)$ with $h=r_{s}1_{\textbf{t}(h)}r_{t}$ hence $h\in R$. 
\end{proof}

In view of this theorem, it is harmless to restrict attention to \emph{canonical} $G$-valuation, i.e. $G$-valuations which satisfy
\begin{enumerate}
	\setcounter{enumi}{\value{enumcounter}}
	\item $v(x)\geq v(\textbf{t}(x))\quad\Leftrightarrow\quad v(x)\geq v(\textbf{s}(x))$.
\end{enumerate}
in addition to the previously mentioned conditions. From now on, we will assume for the sake of simplicity that all $G$-valuations are canonical.

\begin{cor}
In the same context as \ref{Gval}, we have
\begin{align*}
	\left\{h\in H(R)\mid h^{-1}\notin R\right\}&=\left\{h\in  H(R)\mid v(h)>v(\textbf{t}(h))\right\}.
\end{align*}
\end{cor}
\begin{proof}
Take a homogeneous $h$ with $v(h)>v(\textbf{t}(h))$, then 
\begin{align*}
	v(h^{-1})=v(\textbf{t}(h))v(h^{-1})< v(h)v(h^{-1})=v(\textbf{s}(h))=v(\textbf{t}(h^{-1}))
\end{align*} 
so $h^{-1}\notin R$. On the other hand, if $h\in H(R)$ and $h^{-1}\notin R$, then $v(h^{-1})<v(\textbf{t}(h^{-1}))=v(\textbf{s}(h))$ so \[v(\textbf{t}(h))=v(h^{-1})v(h)<v(\textbf{s}(h))v(h)=v(h).\]
\end{proof}

\noindent This set will be denoted by the $P$ of positive. As in the classical case, a $G$-valuation ring is completely determined by its set of positives. Indeed, $R$ is the ring generated by $\left\{h\in H(Q)\mid hP\subseteq P\right\}$. If $R$ is a $G$-valuation ring in a $G$-skewfield $Q$, then $R/P$ inherits a canonical $G$-grading and, in view of the preceding corollary, it will again be a $G$-skewfield, $\overline{Q}$ say. The map \[\pi:Q\cup\left\{\infty\right\}\to\overline{Q}\left\{\infty\right\}:x\mapsto 
\begin{cases}
	\overline{x}&\text{if $x\in R$}\\
	\infty&\text{otherwise}
\end{cases}\]
can then be reasonably be called a \emph{$G$-place} of $Q$ in $\overline{Q}$. In the classical case, there is a one-one correspondence between valuation ring and places (cfr. e.g. \cite{Endler}). No doubt, this could be generalised to the $G$-graded context as well. For the sake of conciseness we will not go into this here.

Let $\Gamma$ and $\Delta$ be ordered groupoids. A bijection $f:\Gamma\to\Delta$ is an \emph{order-preserving isomorphism} if 
\begin{enumerate}
	\item $\forall \gamma,\gamma',\gamma''\in\Gamma:\gamma\gamma'=\gamma''\Leftrightarrow f(\gamma)f(\gamma')=f(\gamma'')$,
	\item $\forall \gamma,\gamma'\in\Gamma\,\forall\delta,\delta'\in\Delta:\gamma\leq\gamma'\Leftrightarrow f(\gamma)\leq f(\gamma')$.
\end{enumerate} 
Two $G$-valuations $v:Q\to\Gamma$ and $w:Q\to\Delta$ are called \emph{equivalent} if there exists an order-preserving isomorphism $f:\Gamma\to\Delta$ with $v(h)=f(w(h))$ for every homogeneous $h$.

\begin{prop}
Two $G$-valuations $v:Q\to\Gamma$ and $w:Q\to\Delta$ are equivalent if and only if $T_{v}=T_{w}$.
\end{prop}
\begin{proof}
Suppose $f$ is an order-preserving isomorphism with $v(h)=f(w(h))$. Then $h\in T_{v}$ if $v(h)\geq v(\textbf{t}(h))$ i.e. $f(w(h))\geq f(w(\textbf{t}(h)))$ which happens precisely if $w(h)\geq w(\textbf{t}(h))$, or in other words, when $h\in T_{w}$.

If, on the other hand, $T_{v}=T_{w}$, then we can define $f:\Gamma\to\Delta:\gamma\mapsto w(v^{-1}(\gamma))$. We must first check that this is indeed a function, so suppose $\gamma=v(h)=v(h')$, then there are some $r_{s},r_{t}\in H(T_{v})^{*}$ with $h=r_{s}h'r_{t}$. These $r_{s},r_{t}$ must necessarily be in $H(T_{w})^{*}$ whence $w(h)=w(r_{s})w(h')w(r_{t})=w(h')$, so $f$ is well-defined. Essentially the same argument proves injectivity and surjectivity we get for free because $w$ is surjective. Moreover, we have $f(\gamma\gamma')= w(v^{-1}(\gamma\gamma'))=w(v^{-1}(\gamma)v^{-1}(\gamma'))$ since $v$ is a $G$-valuation. The right hand side of the last equality is in turn equal to $w(v^{-1}(\gamma))w(v^{-1}(\gamma'))$ which is $f(\gamma)f(\gamma')$ as had to be proven. Finally, if $v(x)=\gamma\leq\gamma'=v(x')$, then there are $r_{s},r_{t}\in H(T_{v})$ with $r_{s}xr_{t}=x'$. Since $T_{v}=T_{w}$, we have $w(x')=w(r_{s}xr_{t})\geq w(x)$. Consequently, $f$ is an order-preserving isomorphism hence $v$ and $w$ are equivalent. 
\end{proof}

\begin{prop}\label{Itsagroup}
If $\mathrm{supp}(Q)$ is connected and equals $\mathrm{supp}(R)$ then the $G$-valuation from \ref{Gval} takes values in a group .  
\end{prop}
\begin{proof}
Suppose $\mathrm{supp}(Q)$ is connected and equal to $\mathrm{supp}(R)$. Take $e,e'$ in $G_{0}$, then there are $r,r'\in R$ with $1_{e}=r1_{e'}r'$, so $\overline{1_{e}}=\overline{1_{e'}}$. Consequently, $\Omega$ has but one idempotent, i.e. it is a group. An ordered groupoid which is a group is a totally ordered group, so $\overline{G}$ is the trivial group whence $\Gamma$ is a group.
\end{proof}

\begin{ex}\label{MoreExamples}
Let us first consider an example of the simplest kind: the $\Delta_{2}$-valuation ring \[\begin{pmatrix}
		R_{v}&R_{v}\\R_{v}&R_{v}
	\end{pmatrix}\quad\text{contained in the $\Delta_{2}$-skewfield}\quad\begin{pmatrix}k&k\\k&k\end{pmatrix}\] for a valuation ring $R_{v}$ with maximal ideal $P$ in a field $k$. In this case, $\overline{G}$ is the trivial group and $\Omega\simeq\Gamma\simeq R_{v}/P$. The associated value function is \[v:Q\to\Gamma\cup\left\{\infty\right\}:\begin{pmatrix}a&b\\c&d\end{pmatrix}\mapsto\min\left\{v(a),v(b),v(c),v(d)\right\}.\]
	This and similar value functions have been studied in \cite{theOtherBook}. Whether a deeper connection exists between groupoid valuations and the value functions considered there would be an interesting topic for future research.   

Matters get a bit more complicated if we consider the situation from example \ref{Gvalex}. Here, $\overline{\Delta_{2}}=\Delta_{2}$ and we find that $\Omega\simeq\mathcal{E}\left[\Delta_{2}\right]$ where $\mathcal{E}$ is the value group of the valuation $v$. For example, $\overline{1_{1,1}}$ and $\overline{1_{2,2}}$ are incomparable in $\Omega$, while $\overline{1_{1,2}}$ is larger and $\overline{1_{2,1}}$ is smaller than both. Let $a$ and $b$ be in $k$ with $v(a)>v(b)$. Then we have, if we denote by some abuse of notation the valuation on the $G$-skewfield with $v$ as well, \[v\left(\begin{pmatrix}b&a\\b&b\end{pmatrix}\right)\geq v\left(\begin{pmatrix}a&b\\a&a\end{pmatrix}\right)\text{ while e.g. }v\left(\begin{pmatrix}a&b\\b&b\end{pmatrix}\right)\text{ and }v\left(\begin{pmatrix}b&b\\b&a\end{pmatrix}\right)\] are incomparable.
\end{ex}


Consider a simple artinian $Q$ which is finite dimensional over its centre $Z(Q)$ and a complete discrete valuation ring $R$ on $Z(Q)$. By the Artin-Wedderburn theorem, we have $Q\simeq M_{n}(D)$ for some skewfield $D$ which is finite dimensional over $Z(Q)$. It is known (cfr. \cite{MaxOrd}) that $R$ is contained in a unique maximal order $\mathcal{O}$ of $D$ and that any maximal order in $M_{n}(D)$ is of the form $qM_{n}(\mathcal{O})q^{-1}$ for some invertible $q$.

\begin{lem}
If $R$ is a $G$-valuation ring on a $G$-skewfield $Q$ and if $q\in U(Q)$, then $qRq^{-1}$ is a $G$-valuation ring as well. 
\end{lem}
\begin{proof}
Suppose $h\in H(Q)\setminus qRq^{-1}$. Then $q^{-1}hq\notin R$ so its $G$-inverse, which is $q^{-1}h^{-1}q$, must be in $R$. Consequently, $h^{-1}\in qRq^{-1}$. On the other hand, $hqR_{\textbf{t}(h)}q^{-1}h^{-1}=hqh^{-1}R_{\textbf{s}(h)}hq^{-1}h^{-1}$. This is furthermore equal to $R_{\textbf{s}(h)}$ since $hqh^{-1}$ has $\textbf{s}(hqh^{-1})=\textbf{t}(hqh^{-1})=\textbf{s}(h)$.
\end{proof}

\begin{rem}
It might be worth pointing out that this $qRq^{-1}$ is a $G$-valuation ring for a different $G$-grading. Indeed, the homogeneous elements will be of the form $qhq^{-1}$ now, where $h$ is homogeneous with respect to the original $G$-grading..
\end{rem}

\noindent We find that, in this case at least, any maximal order is a $G$-valuation ring for a suitable $G$-grading. Note that by \ref{Itsagroup} the associated value function takes values in a group. It seems doubtful that completeness and discreteness are really necessary, which inspires the following question:

\begin{ques}
Is every maximal order a $G$-valuation ring?
\end{ques} 


\section{$G$-valuations and Dubrovin valuation rings}
One of the most important concepts in non-commutative valuation theory is the Dubrovin valuation ring. These rings were introduced by Dubrovin in the eighties and  have been studied quite extensively, in no small part due to their excellent extension properties. For the general theory of Dubrovin valuation rings we refer the interested reader to, for example, \cite{MMU} or \cite{theBook}. In this section we establish a tentative connection between Dubrovin valuation rings and $G$-valuation rings, but there is still work to be done here.
\begin{dfn}
Recall that a subring $R$ of a simple artinian ring $Q$ is called a \emph{Dubrovin valuation ring} if 
\begin{enumerate}
	\item $R/J(R)$ is simple artinian
	\item for every $q\notin R$ there are $r,r'$ in $R$ such that both $rq$ and $qr'$ are in $R\setminus J(R)$.
\end{enumerate}
where $J(R)$ denotes the Jacobson radical of $R$.
\end{dfn}

It is clear from e.g. example \ref{Gvalex} that not every $G$-valuation ring is a Dubrovin valuation ring. The reason is that the ring under consideration there does not have full support, as the following proposition shows.

\begin{prop}\label{GVisDVinM}
If $Q\simeq M_{n}(D)$  for a skewfield $D$ and $R$ is a $\Delta_{n}$-valuation ring containing all $1_{k}\delta_{ij}$, then $R$ is a Dubrovin valuation ring.
\end{prop}
\begin{proof}
By \ref{MaxId}, the ideal $M$ generated by homogeneous elements of $R$ which are not in $R^{*}$ is the unique maximal ideal which does not contain $1_{ii}$ for any $i$. Consider some ideal $I$ containing some $1_{ii}$. Because $R$ contains all $1_{ij}$, it follows that $1_{jj}\in I$ for any $j$. Therefore $I$ must be $R$, so $M$ is maximal. We have $R/M\simeq\bigoplus_{\delta\in\Delta_{n}}R_{\delta}/M_{\delta}$, so $R/M$ is simple Artinian.

Let $v$ be the $G$-valuation as constructed in \ref{Gval}. Note that, by \ref{Itsagroup}, $v$ takes values in a group. If $a=\sum a_{\delta}\delta$ is not in $R$, then there is some $\delta$ with $v(a_{\delta}\delta)<0$ minimal. We find 
\begin{align*}
	v(a(a_{\delta}\delta^{-1}))&=\min_{\gamma\in\Delta_{n}}v(a_{\gamma}\gamma)v((a_{\delta}\delta)^{-1})\\
	&=v(a_{\delta}\delta)v((a_{\delta}\delta)^{-1})=v(1_{l(\delta)})
\end{align*}
which implies that $a(a_{\delta}\delta)^{-1}$ is in $R\setminus M$. In a similar fashion we find an $r$ with $ra\in R\setminus M$, so $R$ is a Dubrovin valuation ring.
\end{proof}

\noindent This suggests the following question: is every Dubrovin valuation ring a $G$-valuation ring? It is known that the property of being a Dubrovin valuation ring is invariant under Morita equivalence, so it would suffice to answer this question for Dubrovin valuation rings in skewfields, i.e.
\begin{ques}
Is every Dubrovin valuation ring in a skewfield a $G$-valuation ring (for some suitable grading)?
\end{ques} 
\noindent In \cite{H&M}, it was shown that this is certainly true in sufficiently nice cases, e.g. if the skewfield is a crossed product and the Dubrovin valuation ring lies above an unramified valuation. Moreover, it is known that the set of divisorial ideals of a Dubrovin valuation ring $R$ forms a groupoid (cfr. \cite{theBook}). This groupoid probably induces a more or less canonical grading on the Ore localisation $Q$ of $R$ with respect to which $Q$ will be a $G$-skewfield. It is to be expected that the Dubrovin valuation ring will then be a $G$-valuation ring in this $G$-skewfield.

\section*{Acknowledgement}

Some of this research was carried out during the author's doctoral studies at the University of Antwerp under supervision of Freddy Van Oystaeyen. The author wants to thank Arno Fehm for his helpful suggestions concerning the presentation of the material.

\end{document}